\documentclass{amsart}%
\usepackage{amssymb}
\usepackage{amsfonts}
\usepackage{amsmath}
\usepackage{epsfig}
\usepackage{graphicx}%
\newtheorem{theorem}{Theorem}
\theoremstyle{plain}

\newtheorem{lemma}{Lemma}
\newtheorem{notation}{Notation}

\newtheorem{proposition}{Proposition}
\newtheorem{remark}{Remark}

\numberwithin{equation}{section}
\begin{document}
\title[a $p-$adic Nagumo type equation]{Local well-posedness of the Cauchy problem for a $p-$adic Nagumo-type equation}
\author[Chac\'{o}n-Cort\'{e}s]{L. F. Chac\'{o}n-Cort\'{e}s}
\address{Pontificia Universidad Javeriana, Departamento de Matem\'{a}ticas, Cra. 7 N.
40-62, Bogot\'{a} D.C., Colombia}
\email{leonardo.chacon@javeriana.edu.co}
\author[Garcia-Bibiano]{C. A. Garcia-Bibiano}
\address{Centro de Investigaci\'{o}n y de Estudios Avanzados del Instituto
Polit\'{e}cnico Nacional. Departamento de Matem\'{a}ticas, Unidad
Quer\'{e}taro. Libramiento Norponiente \#2000, Fracc. Real de Juriquilla.
Santiago de Quer\'{e}taro, Qro. 76230. M\'{e}xico}
\email{cagarcia@math.cinvestav.mx}
\author[Z\'{u}\~{n}iga-Galindo]{W. A. Z\'{u}\~{n}iga-Galindo$^{1}$}
\address{University of Texas Rio Grande Valley. School of Mathematical \& Statistical
Sciences. One West University Blvd. Brownsville, TX 78520, United States}
\email{wilson.zunigagalindo@utrgv.edu}
\thanks{The third author was partially supported by the Lokenath Debnath Endowed
Professorship, UTRGV}

\begin{abstract}
We introduce a new family of $p-$adic non-linear evolution equations. We
establish the local well-posedness of the Cauchy problem for these equations
in Sobolev-type spaces. For a certain subfamily, we show that the blow-up
phenomenon occurs and provide numerical simulations showing this phenomenon.

\end{abstract}
\maketitle

\section{Introduction}

Nowadays, the theory of linear partial pseudo-differential equations for
complex-valued functions over $p-$adic fields is a well-established branch of
mathematical analysis, see e.g. \cite{A-K-S}-\cite{Chacon-Zuniga 2},
\cite{Haran}-\cite{Kochubei pacif}, \cite{Olesko-Khrennikov}%
-\cite{Rodriguez-Zuniga}, \cite{Torresblanca-Zuniga 1}%
-\cite{Zuniga-Nonlinearity}, and references therein. Meanwhile very little is
known about nonlinear $p-$adic equations. We can mention some semilinear
evolution equations solved using $p-$adic wavelets \cite{A-K-S},
\cite{Pourhadi et al}, a kind of equations of reaction-diffusion type and
Turing patterns studied in \cite{Zuniga-JMAA}, \cite{Zuniga-Nonlinearity}, a
$p-$adic analog of one of the porous medium equation
\cite{Khrennikov-Kochubei}, \cite{Olesko-Khrennikov}, the blow-up phenomenon
studied in \cite{Chacon et al}, and non-linear integro-differential equations
connected with $p-$adic cellular networks \cite{Zuniga-1}.

In this article we introduce a new family of nonlinear evolution equations
that we have named as $p-$adic Nagumo-type equations:%
\[
u_{t}=-\gamma\boldsymbol{D}_{x}^{\alpha}u-u^{3}+\left(  \beta+1\right)
u^{2}-\beta u+P(\boldsymbol{D}_{x})\left(  u^{m}\right)  \text{, }%
x\in\mathbb{Q}_{p}^{n},\ t\in\left[  0,T\right]  ,
\]
where $\gamma>0$, $\beta\geq0$, $\boldsymbol{D}_{x}^{\alpha}$, $\alpha>0$, is
the Taibleson operator, $m$ is a positive integer and $P(\boldsymbol{D}_{x})$
is an operator of degree $\delta$ of the form $P(\boldsymbol{D})=\sum
_{j=0}^{k}C_{j}\boldsymbol{D}^{\delta_{j}}$, where the $C_{j}\in\mathbb{R}$
and $\delta_{k}=\delta$. We establish the local well-posedness of the Cauchy
problem for these equations in Sobolev-type spaces, see Theorem \ref{Thorem1}.
For a certain subfamily, we show that the blow-up phenomenon occurs, \ see
Theorem \ref{Thorem2}, and we also provide numerical simulations showing this phenomenon.

The theory of Sobolev-type spaces use here was developed in \cite{Zuniga-JFAA}%
, see also \cite{Rodriguez-Zuniga}, \cite{KKZuniga}. This theory is based in
the theory of countably Hilbert spaces of Gel'fand-Vilenkin
\cite{Gelfand-Vilenkin}. Some generalizations are presented in \cite{Gorka et
al 1}-\cite{Gorka et al 2}. We use classical techniques of operator
semigroups, see e.g. \cite{C-H}, \cite{Milan}. The family of evolution
equations studied here contains as a particular case, equations of the form%
\begin{equation}
u_{t}=-\gamma\boldsymbol{D}_{x}^{\alpha}u-u^{3}+\left(  \beta+1\right)
u^{2}-\beta u, \label{p-adic}%
\end{equation}
where $x\in\mathbb{Q}_{p}^{n},\ t\in\left[  0,T\right]  $, $\boldsymbol{D}%
_{x}^{\alpha}$ is the Taibleson operator, that resemble the classical
Nagumo-type equations, see e.g. \cite{Nagumo et al}.

In \cite{De la Cruz}, the authors study the equations%
\begin{equation}
u_{t}=Du_{xx}-u\left(  u-\kappa\right)  \left(  u-1\right)  -\varepsilon
u_{x}^{m}, \label{real}%
\end{equation}
where $D>0$, $\kappa\in\left(  0,\frac{1}{2}\right)  $, $\varepsilon>0$,
$x\in\mathbb{R}$, $t>0$. They establish the local well-posedness of the Cauchy
problem for these equations in standard Sobolev spaces. There are several
crucial differences between (\ref{p-adic})\ and (\ref{real}). The operators
$u_{xx}$, $u_{x}^{m}$ are local while the operators $\boldsymbol{D}%
_{x}^{\alpha}$, $P(\boldsymbol{D}_{x})\left(  \cdot^{m}\right)  $ are
non-local. The $p-$adic heat equation $u_{t}=-\gamma\boldsymbol{D}_{x}%
^{\alpha}u$ has an arbitrary order of pseudo-differentiability $\alpha>0$ in
the spatial variable, while in the classical fractional heat equation
$u_{t}=D\frac{\partial^{\mu}u}{\partial x^{\mu}}$, the degree of
pseudo-differentiability $\mu\in\left(  0,2\right]  $. This implies that the
Markov processes attached to $u_{t}=-\gamma\boldsymbol{D}_{x}^{\alpha}u$ are
completely different to the ones attached to $u_{t}=Du_{xx}$. In other words,
the diffusion mechanisms in (\ref{p-adic})\ and (\ref{real}) are completely
different. Notice that our non-linear term involves pseudo-derivatives of
arbitrary order $P(\boldsymbol{D}_{x})\left(  u^{m}\right)  $, while in
\cite{De la Cruz} only of first order $u_{x}^{m}$. \ Of course, the $p-$adic
Sobolev spaces behave completely different from their real counterparts, but
the semigroup techniques are the same in both cases, since time is a
non-negative real variable.

The article is organized as follows. In section \ref{Section_2}, we review
some basic aspects of the $p-$adic analysis and fix the notation. In section
\ref{Section_3}, we present some technical results about Sobolev-type spaces
and $p-$adic pseudo-differential operators. In section \ref{Section_4}, we
show the local well-posedness of the $p-$adic Nagumo-type equations, see
Theorem \ref{Thorem1}. In section \ref{Section_5}, we show a subfamily of
$p-$adic Nagumo-type equations whose solutions blow-up in finite time, see
Theorem \ref{Thorem2}. In section \ref{Section_6}, we present a numerical
simulation showing the blow-up phenomenon.

\section{\label{Section_2}$p-$Adic Analysis: Essential Ideas}

In this section, we collect some basic results on $p-$adic analysis that we
use through the article. For a detailed exposition the reader may consult
\cite{A-K-S}, \cite{Kochubei}, \cite{Taibleson}, \cite{V-V-Z}.

\subsection{The field of $p-$adic numbers}

Along this article $p$ will denote a prime number. The field of $p-$adic
numbers $%
\mathbb{Q}
_{p}$ is defined as the completion of the field of rational numbers
$\mathbb{Q}$ with respect to the $p-$adic norm $|\cdot|_{p}$, which is defined
as
\[
\left\vert x\right\vert _{p}=\left\{
\begin{array}
[c]{lll}%
0 & \text{if} & x=0\\
&  & \\
p^{-\gamma} & \text{if} & x=p^{\gamma}\frac{a}{b}\text{,}%
\end{array}
\right.
\]
where $a$ and $b$ are integers coprime with $p$. The integer $\gamma:=ord(x)$,
with $ord(0):=+\infty$, is called the\textit{\ }$p-$\textit{adic order of} $x$.

Any $p-$adic number $x\neq0$ has a unique expansion of the form
\[
x=p^{ord(x)}\sum_{j=0}^{\infty}x_{j}p^{j},
\]
where $x_{j}\in\{0,\dots,p-1\}$ and $x_{0}\neq0$. By using this expansion, we
define \textit{the fractional part of }$x\in\mathbb{Q}_{p}$, denoted
$\{x\}_{p}$, as the rational number
\[
\left\{  x\right\}  _{p}=\left\{
\begin{array}
[c]{lll}%
0 & \text{if} & x=0\text{ or }ord(x)\geq0\\
&  & \\
p^{ord(x)}\sum_{j=0}^{-ord_{p}(x)-1}x_{j}p^{j} & \text{if} & ord(x)<0.
\end{array}
\right.
\]

\subsection{Topology of $\mathbb{Q}_{p}^{n}$}

For $r\in\mathbb{Z}$, denote by $B_{r}^{n}(a)=\{x\in\mathbb{Q}_{p}%
^{n};||x-a||_{p}\leq p^{r}\}$ \textit{the ball of radius }$p^{r}$ \textit{with
center at} $a=(a_{1},\dots,a_{n})\in\mathbb{Q}_{p}^{n}$, and take $B_{r}%
^{n}(0):=B_{r}^{n}$. Note that $B_{r}^{n}(a)=B_{r}(a_{1})\times\cdots\times
B_{r}(a_{n})$, where $B_{r}(a_{i}):=\{x_{i}\in\mathbb{Q}_{p};|x_{i}-a_{i}%
|_{p}\leq p^{r}\}$ is the one-dimensional ball of radius $p^{r}$ with center
at $a_{i}\in\mathbb{Q}_{p}$. The ball $B_{0}^{n}$ equals the product of $n$
copies of $B_{0}=\mathbb{Z}_{p}$, \textit{the ring of }$p-$\textit{adic
integers}. We also denote by $S_{r}^{n}(a)=\{x\in\mathbb{Q}_{p}^{n}%
;||x-a||_{p}=p^{r}\}$ \textit{the sphere of radius }$p^{r}$ \textit{with
center at} $a=(a_{1},\dots,a_{n})\in\mathbb{Q}_{p}^{n}$, and take $S_{r}%
^{n}(0):=S_{r}^{n}$. We notice that $S_{0}^{1}=\mathbb{Z}_{p}^{\times}$ (the
group of units of $\mathbb{Z}_{p}$), but $\left(  \mathbb{Z}_{p}^{\times
}\right)  ^{n}\subsetneq S_{0}^{n}$. The balls and spheres are both open and
closed subsets in $\mathbb{Q}_{p}^{n}$. In addition, two balls in
$\mathbb{Q}_{p}^{n}$ are either disjoint or one is contained in the other.

As a topological space $\left(  \mathbb{Q}_{p}^{n},||\cdot||_{p}\right)  $ is
totally disconnected, i.e. the only connected \ subsets of $\mathbb{Q}_{p}%
^{n}$ are the empty set and the points. A subset of $\mathbb{Q}_{p}^{n}$ is
compact if and only if it is closed and bounded in $\mathbb{Q}_{p}^{n}$, see
e.g. \cite[Section 1.3]{V-V-Z}, or \cite[Section 1.8]{A-K-S}. The balls and
spheres are compact subsets. Thus $\left(  \mathbb{Q}_{p}^{n},||\cdot
||_{p}\right)  $ is a locally compact topological space.

Since $(\mathbb{Q}_{p}^{n},+)$ is a locally compact topological group, there
exists a Haar measure $d^{n}x$, which is invariant under translations, i.e.
$d^{n}(x+a)=d^{n}x$. If we normalize this measure by the condition
$\int_{\mathbb{Z}_{p}^{n}}dx=1$, then $d^{n}x$ is unique.

\begin{notation}
We will use $\Omega\left(  p^{-r}||x-a||_{p}\right)  $ to denote the
characteristic function of the ball $B_{r}^{n}(a)$. For more general sets, we
will use the notation $1_{A}$ for the characteristic function of a set $A$.
\end{notation}

\subsection{The Bruhat-Schwartz space}

A complex-valued function $\varphi$ defined on $\mathbb{Q}_{p}^{n}$ is
\textit{called locally constant} if for any $x\in\mathbb{Q}_{p}^{n}$ there
exist an integer $l(x)\in\mathbb{Z}$ such that%
\begin{equation}
\varphi(x+x^{\prime})=\varphi(x)\text{ for any }x^{\prime}\in B_{l(x)}^{n}.
\label{local_constancy}%
\end{equation}
A function $\varphi:\mathbb{Q}_{p}^{n}\rightarrow\mathbb{C}$ is called a
\textit{Bruhat-Schwartz function (or a test function)} if it is locally
constant with compact support. Any test function can be represented as a
linear combination, with complex coefficients, of characteristic functions of
balls. The $\mathbb{C}-$vector space of Bruhat-Schwartz functions is denoted
by $\mathcal{D}(\mathbb{Q}_{p}^{n}):=\mathcal{D}$. We denote by $\mathcal{D}%
_{\mathbb{R}}(\mathbb{Q}_{p}^{n}):=\mathcal{D}_{\mathbb{R}}$\ the
$\mathbb{R}-$vector space of Bruhat-Schwartz functions. For $\varphi
\in\mathcal{D}(\mathbb{Q}_{p}^{n})$, the largest number $l=l(\varphi)$
satisfying (\ref{local_constancy}) is called \textit{the exponent of local
constancy (or the parameter of constancy) of} $\varphi$.

We denote by $\mathcal{D}_{m}^{l}(\mathbb{Q}_{p}^{n})$ the finite-dimensional
space of test functions from $\mathcal{D}(\mathbb{Q}_{p}^{n})$ having supports
in the ball $B_{m}^{n}$ and with parameters \ of constancy $\geq l$. We now
define a topology on $\mathcal{D}$ as follows. We say that a sequence
$\left\{  \varphi_{j}\right\}  _{j\in\mathbb{N}}$ of functions in
$\mathcal{D}$ converges to zero, if the two following conditions hold:

(1) there are two fixed integers $k_{0}$ and $m_{0}$ such that \ each
$\varphi_{j}\in$ $\mathcal{D}_{m_{0}}^{k_{0}}$;

(2) $\varphi_{j}\rightarrow0$ uniformly.

$\mathcal{D}$ endowed with the above topology becomes a topological vector space.

\subsection{$L^{\rho}$ spaces}

Given $\rho\in\lbrack1,\infty)$, we denote by $L^{\rho}:=L^{\rho}\left(
\mathbb{Q}
_{p}^{n}\right)  :=L^{\rho}\left(
\mathbb{Q}
_{p}^{n},d^{n}x\right)  ,$ the $\mathbb{C}-$vector space of all the
complex-valued functions $g$ satisfying
\[%
{\displaystyle\int\limits_{\mathbb{Q} _{p}^{n}}}
\left\vert g\left(  x\right)  \right\vert ^{\rho}d^{n}x<\infty.
\]
The corresponding $\mathbb{R}-$vector spaces are denoted as $L_{\mathbb{R}%
}^{\rho}\allowbreak:=L_{\mathbb{R}}^{\rho}\left(
\mathbb{Q}
_{p}^{n}\right)  =L_{\mathbb{R}}^{\rho}\left(
\mathbb{Q}
_{p}^{n},d^{n}x\right)  $, $1\leq\rho<\infty$.

If $U$ is an open subset of $\mathbb{Q}_{p}^{n}$, $\mathcal{D}(U)$ denotes the
space of test functions with supports contained in $U$, then $\mathcal{D}(U)$
is dense in
\[
L^{\rho}\left(  U\right)  =\left\{  \varphi:U\rightarrow\mathbb{C};\left\Vert
\varphi\right\Vert _{\rho}=\left\{  \int\limits_{U}\left\vert \varphi\left(
x\right)  \right\vert ^{\rho}d^{n}x\right\}  ^{\frac{1}{\rho}}<\infty\right\}
,
\]
where $d^{n}x$ is the normalized Haar measure on $\left(  \mathbb{Q}_{p}%
^{n},+\right)  $, for $1\leq\rho<\infty$, see e.g. \cite[Section 4.3]{A-K-S}.
We denote by $L_{\mathbb{R}}^{\rho}\left(  U\right)  $ the real counterpart of
$L^{\rho}\left(  U\right)  $.

\subsection{The Fourier transform}

Set $\chi_{p}(y)=\exp(2\pi i\{y\}_{p})$ for $y\in\mathbb{Q}_{p}$. The map
$\chi_{p}(\cdot)$ is an additive character on $\mathbb{Q}_{p}$, i.e. a
continuous map from $\left(  \mathbb{Q}_{p},+\right)  $ into $S$ (the unit
circle considered as multiplicative group) satisfying $\chi_{p}(x_{0}%
+x_{1})=\chi_{p}(x_{0})\chi_{p}(x_{1})$, $x_{0},x_{1}\in\mathbb{Q}_{p}$. \ The
additive characters of $\mathbb{Q}_{p}$ form an Abelian group which is
isomorphic to $\left(  \mathbb{Q}_{p},+\right)  $. The isomorphism is given by
$\kappa\rightarrow\chi_{p}(\kappa x)$, see e.g. \cite[Section 2.3]{A-K-S}.

Given $\xi=(\xi_{1},\dots,\xi_{n})$ and $y=(x_{1},\dots,x_{n})\allowbreak
\in\mathbb{Q}_{p}^{n}$, we set $\xi\cdot x:=\sum_{j=1}^{n}\xi_{j}x_{j}$. The
Fourier transform of $\varphi\in\mathcal{D}(\mathbb{Q}_{p}^{n})$ is defined
as
\[
(\mathcal{F}\varphi)(\xi)=%
{\displaystyle\int\limits_{\mathbb{Q} _{p}^{n}}}
\chi_{p}(\xi\cdot x)\varphi(x)d^{n}x\quad\text{for }\xi\in\mathbb{Q}_{p}^{n},
\]
where $d^{n}x$ is the normalized Haar measure on $\mathbb{Q}_{p}^{n}$. The
Fourier transform is a linear isomorphism from $\mathcal{D}(\mathbb{Q}_{p}%
^{n})$ onto itself satisfying
\begin{equation}
(\mathcal{F}(\mathcal{F}\varphi))(\xi)=\varphi(-\xi), \label{Eq_FFT}%
\end{equation}
see e.g. \cite[Section 4.8]{A-K-S}. We will also use the notation
$\mathcal{F}_{x\rightarrow\xi}\varphi$ and $\widehat{\varphi}$\ for the
Fourier transform of $\varphi$.

The Fourier transform extends to $L^{2}$. If $f\in L^{2},$ its Fourier
transform is defined as
\[
(\mathcal{F}f)(\xi)=\lim_{k\rightarrow\infty}\int\limits_{||x||_{p}\leq p^{k}%
}\chi_{p}(\xi\cdot x)f(x)d^{n}x,\quad\text{for }\xi\in%
\mathbb{Q}
_{p}^{n},
\]
where the limit is taken in $L^{2}$. We recall that the Fourier transform is
unitary on $L^{2},$ i.e. $||f||_{2}=||\mathcal{F}f||_{2}$ for $f\in L^{2}$ and
that (\ref{Eq_FFT}) is also valid in $L^{2}$, see e.g. \cite[Chapter III,
Section 2]{Taibleson}.

\subsection{Distributions}

The $\mathbb{C}-$vector space $\mathcal{D}^{\prime}\left(  \mathbb{Q}_{p}%
^{n}\right)  $ $:=\mathcal{D}^{\prime}$ \ of all continuous linear functionals
on $\mathcal{D}(\mathbb{Q}_{p}^{n})$ is called the \textit{Bruhat-Schwartz
space of distributions}. Every linear functional on $\mathcal{D}$ is
continuous, i.e. $\mathcal{D}^{\prime}$\ agrees with the algebraic dual of
$\mathcal{D}$, see e.g. \cite[Chapter 1, VI.3, Lemma]{V-V-Z}. We denote by
$\mathcal{D}_{\mathbb{R}}^{\prime}\left(  \mathbb{Q}_{p}^{n}\right)  $
$:=\mathcal{D}_{\mathbb{R}}^{\prime}$ the dual space of $\mathcal{D}%
_{\mathbb{R}}$.

We endow $\mathcal{D}^{\prime}$ with the weak topology, i.e. a sequence
$\left\{  T_{j}\right\}  _{j\in\mathbb{n}}$ in $\mathcal{D}^{\prime}$
converges to $T$ if $\lim_{j\rightarrow\infty}T_{j}\left(  \varphi\right)
=T\left(  \varphi\right)  $ for any $\varphi\in\mathcal{D}$. \ The map
\[%
\begin{array}
[c]{lll}%
\mathcal{D}^{\prime}\times\mathcal{D} & \rightarrow & \mathbb{C}\\
\left(  T,\varphi\right)  & \rightarrow & T\left(  \varphi\right)
\end{array}
\]
is a bilinear form which is continuous in $T$ and $\varphi$ separately. We
call this map the pairing between $\mathcal{D}^{\prime}$ and $\mathcal{D}$.
From now on we will use $\left(  T,\varphi\right)  $ instead of $T\left(
\varphi\right)  $.

Every $f$\ in $L_{loc}^{1}$ defines a distribution $f\in\mathcal{D}^{\prime
}\left(  \mathbb{Q}_{p}^{n}\right)  $ by the formula
\[
\left(  f,\varphi\right)  =%
{\textstyle\int\limits_{\mathbb{Q}_{p}^{n}}}
f\left(  x\right)  \varphi\left(  x\right)  d^{n}x.
\]
Such distributions are called \textit{regular distributions}. Notice that for
$f$\ $\in L_{\mathbb{R}}^{2}$, $\left(  f,\varphi\right)  =\left\langle
f,\varphi\right\rangle $, where $\left\langle \cdot,\cdot\right\rangle $
denotes the scalar product in $L_{\mathbb{R}}^{2}$.

\subsection{The Fourier transform of a distribution}

The Fourier transform $\mathcal{F}\left[  T\right]  $ of a distribution
$T\in\mathcal{D}^{\prime}\left(  \mathbb{Q}_{p}^{n}\right)  $ is defined by%
\[
\left(  \mathcal{F}\left[  T\right]  ,\varphi\right)  =\left(  T,\mathcal{F}%
\left[  \varphi\right]  \right)  \text{ for all }\varphi\in\mathcal{D}%
(\mathbb{Q}_{p}^{n})\text{.}%
\]
The Fourier transform $T\rightarrow\mathcal{F}\left[  T\right]  $ is a linear
(and continuous) isomorphism from $\mathcal{D}^{\prime}\left(  \mathbb{Q}%
_{p}^{n}\right)  $\ onto $\mathcal{D}^{\prime}\left(  \mathbb{Q}_{p}%
^{n}\right)  $. Furthermore, $T=\mathcal{F}\left[  \mathcal{F}\left[
T\right]  \left(  -\xi\right)  \right]  $.

\section{\label{Section_3}Sobolev-Type Spaces}

The Sobolev-type spaces used here were introduce in \cite{Zuniga-JFAA},
\cite{Rodriguez-Zuniga}. We follow here closely the presentation given in
\cite[Sections 10.1, 10.2]{KKZuniga}.

We set $\left[  \xi\right]  _{p}:=\max\left\{  1,\left\Vert \xi\right\Vert
_{p}\right\}  $ for $\xi=(\xi_{1},\ldots,\xi_{n})\in\mathbb{Q}_{p}^{n}$. Given
$\varphi,\varrho\in\mathcal{D}(\mathbb{Q}_{p}^{n})$ and $s\in\mathbb{R}$, we
define the scalar product:
\[
\langle\varphi,\varrho\rangle_{s}=\int\limits_{\mathbb{Q}_{p}^{n}}\left[
\xi\right]  _{p}^{s}\widehat{\varphi}(\xi)\overline{\widehat{\varrho}(\xi
)}d^{n}\xi,
\]
where the bar denotes the complex conjugate. We also set $\left\Vert
\varphi\right\Vert _{s}^{2}=\langle\varphi,\varphi\rangle_{s}$, and denote by
$\mathcal{H}_{s}:=\mathcal{H}_{s}(\mathbb{Q}_{p}^{n},\mathbb{C})=\mathcal{H}%
_{s}(\mathbb{C})$ the completion of $\mathcal{D}(\mathbb{Q}_{p}^{n})$ with
respect to $\langle\cdot,\cdot\rangle_{s}$. Notice that if $r,s\in\mathbb{R}$,
with $r\leq s$, then $\left\Vert \cdot\right\Vert _{r}\leq\left\Vert
\cdot\right\Vert _{s}$ and $\mathcal{H}_{s}\hookrightarrow\mathcal{H}_{r}$
(continuous embedding). In particular,
\[
\cdots\supset\mathcal{H}_{-2}\supset\mathcal{H}_{-1}\supset\mathcal{H}%
_{0}\supset\mathcal{H}_{1}\supset\mathcal{H}_{2}\cdots,
\]
where $\mathcal{H}_{0}=L^{2}$. We set
\[
\mathcal{H}_{\infty}(\mathbb{Q}_{p}^{n},\mathbb{C})=\mathcal{H}_{\infty
}:={\bigcap_{s\in\mathbb{N}}}\mathcal{H}_{s}.
\]
Since $\mathcal{H}_{[s]+1}\subseteq\mathcal{H}_{s}\subseteq\mathcal{H}_{[s]}$
for $s\in\mathbb{R}_{+}$, where $[\cdot]$ is the integer part function, then
$\mathcal{H}_{\infty}={\bigcap_{s\in\mathbb{R}_{+}}}\mathcal{H}_{s}$. With the
topology induced by the family of seminorms $\{\Vert\cdot\Vert_{l}%
\}_{l\in\mathbb{N}}$, $\mathcal{H}_{\infty}$ becomes a locally convex space,
which is metrizable. Indeed,%
\[
d(f,g):=\max_{l\in\mathbb{N}}\left\{  2^{-l}\frac{\Vert f-g\Vert_{l}}{1+\Vert
f-g\Vert_{l}}\right\}  ,\text{ with }f,g\in\mathcal{H}_{\infty},
\]
is a metric for the topology of $\mathcal{H}_{\infty}$ considered as a convex
topological space. The metric space $\left(  \mathcal{H}_{\infty},d\right)  $
is the completion of the metric space $(\mathcal{D}(\mathbb{Q}_{p}^{n}),d)$,
cf. \cite[Lemma 10.4]{KKZuniga}. Furthermore, $\mathcal{H}_{\infty}\subset
L^{\infty}\cap C^{\textup{unif}}\cap L^{1}\cap L^{2}$, and $\mathcal{H}%
_{\infty}(\mathbb{Q}_{p}^{n},\mathbb{C})$ is continuously embedded in
$C_{0}(\mathbb{Q}_{p}^{n},\mathbb{C})$. This is the non-Archimedean analog of
the Sobolev embedding theorem, cf. \cite[Theorem 10.15 ]{KKZuniga}.

\begin{lemma}
\label{Lemma1}If $s_{1}\leq s\leq s_{2}$, with $s=\theta s_{1}+(1-\theta
)s_{2}$, $0\leq\theta\leq1$, then $\left\Vert f\right\Vert _{s}\leq\left\Vert
f\right\Vert _{s_{1}}^{\theta}\left\Vert f\right\Vert _{s_{2}}^{(1-\theta)}$.
\end{lemma}

\begin{proof}
Take $f\in\mathcal{H}_{s}$, then by using the H\"{o}lder inequality for the
exponents $\frac{1}{q}=\theta,\frac{1}{q^{\prime}}=1-\theta$,
\begin{align*}
\left\Vert f\right\Vert _{s}^{2}  &  =%
{\displaystyle\int\limits_{\mathbb{Q}_{p}^{n}}}
\left[  \xi\right]  _{p}^{s}\left\vert \widehat{f}\left(  \xi\right)
\right\vert ^{2}d^{n}\xi=%
{\displaystyle\int\limits_{\mathbb{Q}_{p}^{n}}}
\left[  \xi\right]  _{p}^{\theta s_{1}+(1-\theta)s_{2}}\left\vert \widehat
{f}\left(  \xi\right)  \right\vert ^{2\left(  \theta+\left(  1-\theta\right)
\right)  }d^{n}\xi\\
&  =%
{\displaystyle\int\limits_{\mathbb{Q}_{p}^{n}}}
\left(  \left[  \xi\right]  _{p}^{s_{1}}\left\vert \widehat{f}\left(
\xi\right)  \right\vert ^{2}\right)  ^{\theta}\left(  \left[  \xi\right]
_{p}^{s_{2}}\left\vert \widehat{f}\left(  \xi\right)  \right\vert ^{2}\right)
^{1-\theta}d^{n}\xi\\
&  \leq\left(
{\displaystyle\int\limits_{\mathbb{Q}_{p}^{n}}}
\left[  \xi\right]  _{p}^{s_{1}}\left\vert \widehat{f}\left(  \xi\right)
\right\vert ^{2}d^{n}\xi\right)  ^{\theta}\left(
{\displaystyle\int\limits_{\mathbb{Q}_{p}^{n}}}
\left[  \xi\right]  _{p}^{s_{2}}\left\vert \widehat{f}\left(  \xi\right)
\right\vert ^{2}d^{n}\xi\right)  ^{1-\theta}d^{n}\xi.
\end{align*}

\end{proof}

The following characterization of the spaces $\mathcal{H}_{s}$ and
$\mathcal{H}_{\infty}$ is useful:

\begin{lemma}
[{\cite[Lemma 10.8]{KKZuniga}}](i) $\mathcal{H}_{s}=\left\{  f\in
L^{2};\left\Vert f\right\Vert _{s}<\infty\right\}  =\left\{  T^{\prime}%
\in\mathcal{D};\left\Vert T\right\Vert _{s}<\infty\right\}  $, (ii)
$\mathcal{H}_{\infty}=\left\{  f\in L^{2};\left\Vert f\right\Vert _{s}%
<\infty\text{ for any }s\in\mathbb{R}_{+}\right\}  =\left\{  T^{\prime}%
\in\mathcal{D};\left\Vert T\right\Vert _{s}<\infty\text{ for any }%
s\in\mathbb{R}_{+}\right\}  $. The equalities in (i)-(ii) are in the sense of
vector spaces.
\end{lemma}

\begin{proposition}
\label{Prop1}If $s>n/2$, then $\mathcal{H}_{s}$ is a Banach algebra with
respect to the product of functions. That is, if $f,g\in\mathcal{H}_{s}$, then
$fg\in\mathcal{H}_{s}$ and $\left\Vert fg\right\Vert _{s}\leq C(n,s)\left\Vert
f\right\Vert _{s}\left\Vert g\right\Vert _{s}$, where $C(n,s)$\ is a positive constant.
\end{proposition}

\begin{proof}
By the ultrametric property of $\left\Vert \cdot\right\Vert _{p}$, $\left\Vert
\xi\right\Vert _{p}\leq\max\left\{  \left\Vert \xi-\eta\right\Vert
_{p},\left\Vert \eta\right\Vert _{p}\right\}  $ for $\xi,\eta\in\mathbb{Q}%
_{p}^{n}$, we have $\max\left\{  1,\left\Vert \xi\right\Vert _{p}\right\}
\leq\max\left\{  1,\left\Vert \xi-\eta\right\Vert _{p},\left\Vert
\eta\right\Vert _{p}\right\}  $, which implies that%
\[
\left[  \max\left\{  1,\left\Vert \xi\right\Vert _{p}\right\}  \right]
^{s}\leq\max\left\{  1,\left\Vert \xi-\eta\right\Vert _{p}^{s},\left\Vert
\eta\right\Vert _{p}^{s}\right\}  =\max\left\{  1,\left\Vert \xi
-\eta\right\Vert _{p},\left\Vert \eta\right\Vert _{p}\right\}  ^{s}%
\]
for $s>0$. Therefore%
\begin{equation}
\left[  \xi\right]  _{p}^{s}\leq\left[  \xi-\eta\right]  _{p}^{s}+\left[
\eta\right]  _{p}^{s}. \label{Eq_10}%
\end{equation}
Now, for $f,g\in L^{2}$, by using (\ref{Eq_10}),%
\begin{align*}
\lbrack\xi]_{p}^{\frac{s}{2}}\left\vert \widehat{fg}\left(  \xi\right)
\right\vert  &  =\left\vert [\xi]_{p}^{\frac{s}{2}}%
{\displaystyle\int\limits_{\mathbb{Q}_{p}^{n}}}
\widehat{f}(\xi-\eta)\widehat{g}(\eta)d^{n}\eta\right\vert \\
&  \leq%
{\displaystyle\int\limits_{\mathbb{Q}_{p}^{n}}}
[\xi-\eta]_{p}^{\frac{s}{2}}\left\vert \widehat{f}(\xi-\eta)\right\vert
\left\vert \widehat{g}(\eta)\right\vert d^{n}\eta+%
{\displaystyle\int\limits_{\mathbb{Q}_{p}^{n}}}
[\eta]_{p}^{\frac{s}{2}}\left\vert \widehat{g}(\eta)\right\vert \left\vert
\widehat{f}(\xi-\eta)\right\vert d^{n}\eta\\
&  =[\xi]_{p}^{\frac{s}{2}}\left\vert \widehat{f}(\xi)\right\vert
\ast\left\vert \widehat{g}(\xi)\right\vert +\left\vert \widehat{f}%
(\xi)\right\vert \ast\lbrack\xi]_{p}^{\frac{s}{2}}\left\vert \widehat{g}%
(\xi)\right\vert .
\end{align*}
Then%
\begin{align*}
\left\Vert fg\right\Vert _{s}  &  \leq\left\Vert \lbrack\xi]_{p}^{\frac{s}{2}%
}\left\vert \widehat{f}(\xi)\right\vert \ast\left\vert \widehat{g}%
(\xi)\right\vert +\left\vert \widehat{f}(\xi)\right\vert \ast\lbrack\xi
]_{p}^{\frac{s}{2}}\left\vert \widehat{g}(\xi)\right\vert \right\Vert _{2}\\
&  \leq\left\Vert \lbrack\xi]_{p}^{\frac{s}{2}}\left\vert \widehat{f}%
(\xi)\right\vert \ast\left\vert \widehat{g}(\xi)\right\vert \right\Vert
_{2}+\left\Vert \left\vert \widehat{f}(\xi)\right\vert \ast\lbrack\xi
]_{p}^{\frac{s}{2}}\left\vert \widehat{g}(\xi)\right\vert \right\Vert _{2}.
\end{align*}
Since $[\xi]_{p}^{\frac{s}{2}}\left\vert \widehat{f}(\xi)\right\vert $,
$[\xi]_{p}^{\frac{s}{2}}\left\vert \widehat{g}(\xi)\right\vert \in L^{2}$, by
using the Cauchy-Schwarz inequality with $s>n/2$, \ we have $\left\Vert
\left\vert \widehat{g}(\xi)\right\vert \right\Vert _{1}\leq A(n,s)\left\Vert
g\right\Vert _{s}$, $\left\Vert \left\vert \widehat{f}(\xi)\right\vert
\right\Vert _{1}\leq A(n,s)\left\Vert f\right\Vert _{s}$, i.e. $\left\vert
\widehat{g}(\xi)\right\vert $, $\left\vert \widehat{f}(\xi)\right\vert \in
L^{1}$. Now, by the Young inequality, we obtain that%
\[
\left\Vert fg\right\Vert _{s}\leq\left\Vert f\right\Vert _{s}\left\Vert
\widehat{g}\right\Vert _{1}+\left\Vert g\right\Vert _{s}\left\Vert \widehat
{f}\right\Vert _{1}\leq2A(n,s)\left\Vert f\right\Vert _{s}\left\Vert
g\right\Vert _{s}.
\]

\end{proof}

\subsection{The Taibleson operator}

Let $\alpha>0$, the Taibleson operator is defined as
\[
(\boldsymbol{D}^{\alpha}\varphi)(x)=\mathcal{F}_{\xi\rightarrow x}%
^{-1}(\left\Vert \xi\right\Vert _{p}^{\alpha}(\mathcal{F}_{x\rightarrow\xi
}\varphi)),
\]
for $\varphi\in\mathcal{D}(\mathbb{Q}_{p}^{n})$. This operator admits the
extension
\[
(\boldsymbol{D}^{\alpha}f)(x)=\frac{1-p^{\alpha}}{1-p^{-\alpha-n}}%
\int\limits_{\mathbb{Q}_{p}^{n}}\left\Vert y\right\Vert _{p}^{-\alpha
-n}\{f(x-y)-f(x)\}d^{n}y
\]
to locally constant functions satisfying
\[
\int\limits_{\left\Vert x\right\Vert _{p}>1}\left\Vert x\right\Vert
_{p}^{-\alpha-n}\left\vert f\left(  x\right)  \right\vert d^{n}x<\infty.
\]
The Taibleson operator $\boldsymbol{D}^{\alpha}$ is the $p-$adic analog of the
fractional derivative. If $n=1$, $\boldsymbol{D}^{\alpha}$ agrees with the
Vladimirov operator. The operator $\boldsymbol{D}^{\alpha}$ does not satisfy
the chain rule neither Leibniz formula. We also use the notation
$\boldsymbol{D}_{x}^{\alpha}$, when the Taibleson operator acts on functions
depending on the variables $x\in\mathbb{Q}_{p}^{n}$ and $t\geq0$.

Given $0=\delta_{0}<\delta_{1}<\cdots<\delta_{k-1}<\delta_{k}=\delta$, we
define%
\[
P(\boldsymbol{D})=%
{\displaystyle\sum\limits_{j=0}^{k}}
C_{j}\boldsymbol{D}^{\delta_{j}}\text{, where the }C_{j}\in\mathbb{R}\text{.}%
\]

\begin{lemma}
[{\cite[Lemma 10.13 and Theorem 10.15]{KKZuniga}}]\label{Lemma2} For
$s\in\mathbb{R}_{+}$, the mapping $P(\boldsymbol{D}):\mathcal{H}_{s+2\delta
}\longrightarrow\mathcal{H}_{s}$ is a well-defined continuous mapping between
Banach spaces.
\end{lemma}

\begin{lemma}
\label{Lemma3}Take $s-2\delta>n/2$ and $f,g\in\mathcal{H}_{s+2\delta}$. Then
\[
\left\Vert P(\boldsymbol{D})\left(  fg\right)  \right\Vert _{s}\leq
C(n,s,\delta)\left\Vert f\right\Vert _{s+2\delta}\left\Vert g\right\Vert
_{s+2\delta},
\]
where $C(n,s,\delta)$ is a positive constant that depends of $n$, $s$ and
$\delta$.
\end{lemma}

\begin{proof}
Since $s>n/2$ and $f,g\in\mathcal{H}_{s+2\delta}$, by Proposition \ref{Prop1},
$fg\in\mathcal{H}_{s+2\delta}$, and by Lemma \ref{Lemma2}, $P(\boldsymbol{D}%
)\left(  fg\right)  \in\mathcal{H}_{s}$. Now by using Proposition
\ref{Prop1},
\begin{gather*}
\left\Vert P(\boldsymbol{D})\left(  fg\right)  \right\Vert _{s}\leq%
{\displaystyle\sum\limits_{j=0}^{k}}
\left\vert C_{j}\right\vert \left\Vert \boldsymbol{D}^{\delta_{j}}\left(
fg\right)  \right\Vert _{s}\\
=%
{\displaystyle\sum\limits_{j=0}^{k}}
\left\vert C_{j}\right\vert \left(
{\displaystyle\int\limits_{\mathbb{Q}_{p}^{n}}}
\left[  \xi\right]  _{p}^{s}\left\Vert \xi\right\Vert _{p}^{2\delta_{j}%
}\left\vert \widehat{fg}\left(  \xi\right)  \right\vert ^{2}d^{n}\xi\right)
^{\frac{1}{2}}\leq%
{\displaystyle\sum\limits_{j=0}^{k}}
\left\vert C_{j}\right\vert \left(
{\displaystyle\int\limits_{\mathbb{Q}_{p}^{n}}}
\left[  \xi\right]  _{p}^{s+2\delta_{j}}\left\vert \widehat{fg}\left(
\xi\right)  \right\vert ^{2}d^{n}\xi\right)  ^{\frac{1}{2}}\\
=%
{\displaystyle\sum\limits_{j=0}^{k}}
\left\vert C_{j}\right\vert \left\Vert fg\right\Vert _{s+2\delta_{j}}\leq%
{\displaystyle\sum\limits_{j=0}^{k}}
\left\vert C_{j}\right\vert C(n,s,\delta_{j})\left\Vert f\right\Vert
_{s+2\delta_{j}}\left\Vert g\right\Vert _{s+2\delta_{j}}\\
\leq\left(
{\displaystyle\sum\limits_{j=0}^{k}}
\left\vert C_{j}\right\vert C(n,s,\delta_{j})\right)  \left\Vert f\right\Vert
_{s+2\delta}\left\Vert g\right\Vert _{s+2\delta}.
\end{gather*}

\end{proof}

\section{\label{Section_4}Local well-posedness of the $p-$adic Nagumo-type
equations}

\subsection{Some technical remarks}

Let $X$, $Y$ Banach spaces, $T_{0}\in(0,\infty)$ and let $F:\left[
0,T_{0}\right]  \times Y\longrightarrow X$ a continuous function. The Cauchy
problem%
\begin{equation}
\left\{
\begin{array}
[c]{l}%
\partial_{t}u(t)=F\left(  t,u\left(  t\right)  \right) \\
\\
u(0)=\phi\in Y
\end{array}
\right.  \label{Eq-11}%
\end{equation}
is locally well-posed in $Y$, if the following conditions are satisfied.

(i) There is $T\in\left(  0,T_{0}\right]  $ and a function $u\in C([0,T];Y)$,
with $u(0)=\phi$, satisfying the differential equation in the following
sense:
\[
\lim_{h\rightarrow0}\left\Vert \frac{u(t+h)-u(t)}{h}-F(t,u(t))\right\Vert
_{X}=0,
\]
where the derivatives at $t=0$ and $t=T$ are calculated from the right and
left, respectively.

(ii) The initial value problem (\ref{Eq-11}) has at most one solution in
$C([0,T];Y)$.

(iii) The function $\phi\rightarrow u$ is continuous. That is, let $\left\{
\phi_{n}\right\}  $ be a sequence in $Y$ such that $\phi_{n}\rightarrow
\phi_{\infty}$ in $Y$ and let $u_{n}\in C\left(  \left[  0,T_{n}\right]
;Y\right)  $, resp. $u_{\infty}\in C\left(  \left[  0,T_{\infty}\right]
;Y\right)  $, be the corresponding solutions. Let $T\in\left(  0,T_{\infty
}\right)  $, then the solutions $u_{n}$ are defined in $[0,T]$ for all $n$ big
enough and
\[
\lim_{n\rightarrow\infty}\sup_{t\in\lbrack0,T]}\left\Vert u_{n}(t)-u_{\infty
}(t)\right\Vert _{Y}=0.
\]

\subsection{Main result}

Consider the following Cauchy problem:%
\begin{equation}
\left\{
\begin{array}
[c]{ll}%
u\in C\left(  \left[  0,T\right]  ,\mathcal{H}_{s}\right)  \cap C^{1}\left(
\left[  0,T\right]  ,\mathcal{H}_{s}\right)  ; & \\
& \\
u_{t}=-\gamma\boldsymbol{D}_{x}^{\alpha}u-u^{3}+\left(  \beta+1\right)
u^{2}-\beta u+P(\boldsymbol{D}_{x})\left(  u^{m}\right)  , & x\in
\mathbb{Q}_{p}^{n},\ t\in\left[  0,T\right]  ;\\
& \\
u(0)=f_{0}\in\mathcal{H}_{s}, &
\end{array}
\right.  \label{Cauchy-Problem}%
\end{equation}
where $T$, $\gamma$, $\alpha$, $\beta>0$, and $m$ is a positive integer. The
main result of this work is the following:

\begin{theorem}
\label{Thorem1}For $s>n/2+2\delta$, the Cauchy problem (\ref{Cauchy-Problem})
is locally well-posed in $\mathcal{H}_{s}$.
\end{theorem}

\subsection{Preliminary results}

We denote by $\boldsymbol{V}(t)=e^{-(\gamma\boldsymbol{D}^{\alpha}%
+\beta\boldsymbol{I})t}$, $t\geq0$, the semigroup in $L^{2}$ generated by the
operator $\boldsymbol{A}=-\gamma\boldsymbol{D}^{\alpha}-\beta\boldsymbol{I}$,
that is,
\[
\boldsymbol{V}(t)f\left(  x\right)  =\mathcal{F}_{\xi\rightarrow x}%
^{-1}\left(  e^{-(\gamma\lVert\xi\rVert_{p}^{\alpha}+\beta)t}\mathcal{F}%
_{x\rightarrow\xi}f\right)  \text{, \textup{for}}\ f\in L^{2},\ t\geq0.
\]

\begin{lemma}
\label{Lemma0}$\{\boldsymbol{V}(t)\}_{t\geq0}$ is a $C^{0}-$semigroup of
contractions in $\mathcal{H}_{s}$, $s\in\mathbb{R}$, satisfying $\left\Vert
\boldsymbol{V}(t)\right\Vert _{s}\leq e^{-\beta t}$ for $t\geq0$. Moreover,
$u(x,t)=\boldsymbol{V}(t)f_{0}(x)$ is the unique solution to the following
Cauchy problem:
\begin{equation}
\left\{
\begin{array}
[c]{l}%
u\in C\left(  \left[  0,T\right]  ,\mathcal{H}_{s}\right)  \cap C^{1}\left(
\left[  0,T\right]  ,\mathcal{H}_{s}\right)  ;\\
\\
u_{t}=-\gamma\boldsymbol{D}^{\alpha}u-\beta u,\text{ }t\in\left[  0,T\right]
;\\
\\
u(x,0)=f_{0}(x)\in\mathcal{H}_{s},
\end{array}
\right.  \label{Cauchy-Problem_0}%
\end{equation}
where $T$ is an arbitrary positive number.
\end{lemma}

\begin{proof}
We just verify the strongly continuity of the semigroup. Since
\begin{multline*}
\left\Vert \mathcal{F}_{\xi\rightarrow x}^{-1}\left(  e^{-(\gamma\lVert
\xi\rVert_{p}^{\alpha}+\beta)t}\mathcal{F}_{x\rightarrow\xi}f\right)
-f\left(  x\right)  \right\Vert _{s}^{2}\\
=%
{\displaystyle\int\limits_{\mathbb{Q}_{p}^{n}}}
\left[  \xi\right]  _{p}^{s}\left\vert \widehat{f}\left(  \xi\right)
\right\vert ^{2}\left\{  1-e^{-(\gamma\lVert\xi\rVert_{p}^{\alpha}+\beta
)t}\right\}  ^{2}d^{n}\xi\leq\left\Vert f\right\Vert _{s}^{2},
\end{multline*}
it follows from the dominated convergence theorem that
\[
\lim_{t\rightarrow0+}\left\Vert \boldsymbol{V}(t)f-f\right\Vert _{s}=0.
\]
The existence and uniqueness of a solution for the Cauchy problem
(\ref{Cauchy-Problem_0}) follows from a well-known result, see e.g.
\cite[Theorem 4.3.1]{Milan}.
\end{proof}

\begin{lemma}
\label{LemmaA}Let $f_{0}\in\mathcal{H}_{s}$, $s\in\mathbb{R}$, $\lambda\geq0$.
Then, there exists a positive constant $C(\lambda,\alpha)$ that depends of
$\lambda$ and $\alpha$ such that
\begin{equation}
\lVert\boldsymbol{V}(t)f_{0}\rVert_{s+\lambda}\leq e^{-\beta t}\left(
1+C(\lambda,\alpha)\left(  \frac{\lambda}{2\alpha\gamma t}\right)
^{\frac{\lambda}{2\alpha}}\right)  \lVert f_{0}\rVert_{s}\text{ \ for
}t>0\text{.} \label{Eq-12}%
\end{equation}

\end{lemma}

\begin{proof}
We first notice that%
\begin{gather*}
\left\Vert \boldsymbol{V}\left(  t\right)  f_{0}\right\Vert _{s+\lambda}%
^{2}=\int\limits_{\mathbb{Q}_{p}^{n}}[\xi]_{p}^{s+\lambda}e^{-2(\gamma
\lVert\xi\rVert_{p}^{\alpha}+\beta)t}\left\vert f_{0}\left(  \xi\right)
\right\vert ^{2}d^{n}\xi\\
\leq e^{-2\beta t}\left(  \sup_{\xi\in\mathbb{Q}_{p}^{n}}[\xi]_{p}^{\lambda
}e^{-2\gamma\lVert\xi\rVert_{p}^{\alpha}t}\right)  \left\Vert f_{0}\right\Vert
_{s}^{2}\leq e^{-2\beta t}\left(  1+\sup_{\xi\in\mathbb{Q}_{p}^{n}%
\smallsetminus\mathbb{Z}_{p}^{n}}\left\Vert \xi\right\Vert _{p}^{\lambda
}e^{-2\gamma\lVert\xi\rVert_{p}^{\alpha}t}\right)  \left\Vert f_{0}\right\Vert
_{s}^{2}\\
\leq e^{-2\beta t}\left(  1+\sup_{\xi\in\mathbb{Q}_{p}^{n}}\left\Vert
\xi\right\Vert _{p}^{\lambda}e^{-2\gamma\lVert\xi\rVert_{p}^{\alpha}t}\right)
\left\Vert f_{0}\right\Vert _{s}^{2}.
\end{gather*}
We now set $y=\lVert\xi\rVert_{p}$ and $h(y)=y^{\lambda}e^{-2\gamma y^{\alpha
}t}$. By using the fact that $h(y)$ reaches its maximum at $y_{\textup{max}%
}=\left(  \frac{\lambda}{2\alpha\gamma t}\right)  ^{\frac{1}{\alpha}}$, we
conclude that%
\[
\sup_{\xi\in\mathbb{Q}_{p}^{n}}\left\Vert \xi\right\Vert _{p}^{\lambda
}e^{-2\gamma\lVert\xi\rVert_{p}^{\alpha}t}\leq\left(  \frac{\lambda}%
{2\alpha\gamma t}\right)  ^{\frac{\lambda}{\alpha}}e^{-\frac{\lambda}{\alpha}%
}\leq\textup{C}\left(  \lambda,\alpha\right)  \left(  \frac{\lambda}%
{2\alpha\gamma t}\right)  ^{\frac{\lambda}{\alpha}}.
\]

\end{proof}

\begin{proposition}
\label{Prop2}Let $s>n/2+2\delta$ and $F(u)=(\beta+1)u^{2}-u^{3}+P\left(
\boldsymbol{D}\right)  (u^{m})$. Then $F:\mathcal{H}_{s}\longrightarrow
\mathcal{H}_{s-2\delta}$ is a continuous function satisfying
\begin{equation}
\lVert F(u)-F(w)\lVert_{s-2\delta}\leq L(\lVert u\lVert_{s},\lVert w\lVert
_{s})\lVert u-w\lVert_{s}, \label{Eq-13}%
\end{equation}
for $u,w\in\mathcal{H}_{s}$, here $L(\cdot,\cdot)$ is a continuous function,
which is not decreasing with respect to each of their arguments. In
particular,
\begin{equation}
\lVert F(u)\rVert_{s-2\delta}\leq L(\lVert u\rVert_{s},0)\lVert u\rVert_{s}.
\label{Eq-14}%
\end{equation}

\end{proposition}

\begin{proof}
We first notice that%
\begin{gather*}
F(u)-F(w)=(\beta+1)(u^{2}-w^{2})-(u^{3}-w^{3})+P\left(  \boldsymbol{D}\right)
(u^{m}-w^{m})\\
=(\beta+1)(u-w)\left(  u+w\right)  -(u-w)(u^{2}+uw+w^{2})+P\left(
\boldsymbol{D}\right)  ((u-w)q(u,w)),
\end{gather*}
where $q(u,w)=\sum_{k=0}^{m-1}u^{k}w^{m-1-k}$. By using Proposition
\ref{Prop1} and Lemma \ref{Lemma3}, the condition $s>n/2$ implies that if
$u,w\in\mathcal{H}_{s}$, then any polynomial function in $u,w$ belongs to
$\mathcal{H}_{s}$, and%
\begin{multline*}
\left\Vert F(u)-F(w)\right\Vert _{s-2\delta}\leq C\left\{  (\beta+1)\lVert
u-w\rVert_{s-2\delta}\lVert u+w\rVert_{s-2\delta}+\right. \\
\left.  \lVert u-w\rVert_{s-2\delta}\lVert u^{2}+uw+w^{2}\rVert_{s-2\delta
}+\lVert u-w\rVert_{s}\left\Vert q(u,w)\right\Vert _{s}\right\}  ,
\end{multline*}
where $C=$ $C(n,s,\delta)$. Then
\[
\left\Vert F(u)-F(w)\right\Vert _{s-2\delta}\leq A(\lVert u\rVert_{s},\lVert
w\rVert_{s})\lVert u-w\lVert_{s},
\]
where%
\begin{multline*}
A(\lVert u\rVert_{s},\lVert w\rVert_{s})=C\left\{  (\beta+1)\lVert
u+w\rVert_{s}+\lVert u^{2}+uw+w^{2}\rVert_{s}+\left\Vert q(u,w)\right\Vert
_{s}\right\} \\
\leq C\left\{  (\beta+1)\lVert u\rVert_{s}+(\beta+1)\lVert w\rVert_{s}+\lVert
u^{2}\rVert_{s}+\lVert uw\rVert_{s}+\lVert w^{2}\rVert_{s}+\sum_{k=0}%
^{m-1}\left\Vert u^{k}w^{m-1-k}\right\Vert _{s}\right\} \\
\leq C(\beta+1)\lVert u\rVert_{s}+C(\beta+1)\lVert w\rVert_{s}+C^{2}\lVert
u\rVert_{s}^{2}+C^{2}\lVert u\rVert_{s}\lVert w\rVert_{s}+C^{2}\lVert
w\rVert_{s}^{2}+\\
C^{m+1}\sum_{k=0}^{m-1}\left\Vert u\right\Vert _{s}^{k}\left\Vert w\right\Vert
_{s}^{m-1-k}=:L(\lVert u\lVert_{s},\lVert w\lVert_{s}).
\end{multline*}

\end{proof}

For $M,T>0$ and $f_{0}\in\mathcal{H}_{s}$, we set
\[
\mathcal{X}(M,T,f_{0}):=\left\{  u(t)\in C\left(  [0,T];\mathcal{H}%
_{s}\right)  ;\sup_{t\in\lbrack0,T]}\Vert u(t)-V(t)f_{0}\Vert_{s}\leq
M\right\}  .
\]
We endow $\mathcal{X}(M,T,f_{0})$ with the metric $d(u(t),v(t))=\sup
_{t\in\lbrack0,T]}\Vert u(t)-v(t)\Vert_{s}$. The resulting metric space is complete.

\begin{proposition}
\label{Prop3}Take $f_{0}\in\mathcal{H}_{s}$ with $s>n/2+2\delta$, $\delta>0$.
Then, there exists $T=T(\lVert f_{0}\rVert_{s},M)>0$ and a unique function
$u\in C([0,T];\mathcal{H}_{s})$ satisfying the integral equation
\begin{equation}
u(t)=\boldsymbol{V}(t)f_{0}+\int\nolimits_{0}^{t}\boldsymbol{V}(t-\tau
)F(u(\tau))d\tau, \label{Mild_Solution}%
\end{equation}
such that $u(0)=f_{0}$. Here $F(u)=(\beta+1)u^{2}-u^{3}+P(\boldsymbol{D}%
)(u^{m})$ as before.
\end{proposition}

\begin{remark}
Since $F(u)$ is not a locally Lipschitz function because inequality
(\ref{Eq-14}) involves two different norms, the existence of mild solutions of
type (\ref{Mild_Solution}) does not follow directly from standard results in
semigroup theory, see e.g. \cite[Theorem 5.2.2]{Milan}.
\end{remark}

\begin{proof}
Given $u\in\mathcal{X}(M,T,f_{0})$, we set
\[
\boldsymbol{N}u(t)=\boldsymbol{V}(t)f_{0}+\int\nolimits_{0}^{t}\boldsymbol{V}%
(t-\tau)F(u(\tau))d\tau.
\]
\textbf{Claim\ 1}. $\boldsymbol{N}:\mathcal{X}(M,T,f_{0})\longrightarrow
C([0,T];\mathcal{H}_{s})$.

Take $u\in\mathcal{X}(M,T,f_{0})$, then%
\begin{gather}
\left\Vert \boldsymbol{N}u\left(  t_{1}\right)  -\boldsymbol{N}u\left(
t_{2}\right)  \right\Vert _{s}\leq\left\Vert \left(  \boldsymbol{V}\left(
t_{1}\right)  -\boldsymbol{V}\left(  t_{2}\right)  \right)  f_{0}\right\Vert
_{s}\label{Eq-15}\\
+\left\Vert \int\nolimits_{0}^{t_{1}}\boldsymbol{V}\left(  t_{1}-\tau\right)
F(u(\tau))d\tau-\int\nolimits_{0}^{t_{2}}\boldsymbol{V}\left(  t_{2}%
-\tau\right)  F(u(\tau))d\tau\right\Vert _{s}.\nonumber
\end{gather}
Since $\{\boldsymbol{V}(t)\}_{t\geq0}$ is a $C_{0}-$semigroup in
$\mathcal{H}_{s}$, cf. Lemma \ref{Lemma0}, the first term on the right-hand
side of the inequality (\ref{Eq-15}) tends to zero when $t_{2}\rightarrow
t_{1}$. To study the second term, we assume without loss of generality that
$0<t_{1}<t_{2}<T$. Then%
\begin{align*}
&  \left\Vert \int\nolimits_{0}^{t_{1}}\boldsymbol{V}\left(  t_{1}%
-\tau\right)  F(u(\tau))d\tau-\int\nolimits_{0}^{t_{2}}\boldsymbol{V}\left(
t_{2}-\tau\right)  F(u(\tau))d\tau\right\Vert _{s}\\
&  \leq\int\nolimits_{0}^{t_{1}}\left\Vert \{\boldsymbol{V}\left(  t_{1}%
-\tau\right)  -\boldsymbol{V}\left(  t_{2}-\tau\right)  \}F(u(\tau
))\right\Vert _{s}d\tau+\int\nolimits_{t_{1}}^{t_{2}}\left\Vert \boldsymbol{V}%
\left(  t_{2}-\tau\right)  F(u(\tau))\right\Vert _{s}d\tau.
\end{align*}
By using Lemma \ref{LemmaA} with $\lambda=\alpha$ and Proposition
\ref{Prop2},
\begin{gather*}
\left\Vert \left(  \boldsymbol{V}\left(  t_{1}-\tau\right)  -\boldsymbol{V}%
\left(  t_{2}-\tau\right)  \right)  F(u(\tau))\right\Vert _{s}\\
\leq\left\Vert \boldsymbol{V}\left(  t_{1}-\tau\right)  F(u(\tau))\right\Vert
_{s}+\left\Vert \boldsymbol{V}\left(  t_{2}-\tau\right)  F(u(\tau))\right\Vert
_{s}\\
\leq\left\{  2+C_{0}\left(  \frac{1}{2\gamma(t_{1}-\tau)}\right)  ^{\frac
{1}{2}}+C_{0}\left(  \frac{1}{2\gamma(t_{2}-\tau)}\right)  ^{\frac{1}{2}%
}\right\}  \Vert F(u(\tau))\Vert_{s-\alpha}\\
\leq2\left\{  1+C_{0}\left(  \frac{1}{2\gamma(t_{1}-\tau)}\right)  ^{\frac
{1}{2}}\right\}  \sup_{\tau\in\lbrack0,T]}\Vert F\left(  u(\tau)\right)
\Vert_{s-\alpha}\\
=A(T,s,\alpha)\left\{  1+C_{0}\left(  \frac{1}{2\gamma(t_{1}-\tau)}\right)
^{\frac{1}{2}}\right\}  \in L^{1}([0,t_{1}]).
\end{gather*}
Now, by applying the dominated convergence theorem,
\[
\lim\limits_{t_{2}\rightarrow t_{1}}\int\nolimits_{0}^{t_{1}}\left\Vert
\left(  \boldsymbol{V}\left(  t_{1}-\tau\right)  -\boldsymbol{V}\left(
t_{2}-\tau\right)  \right)  F(u(\tau))\right\Vert _{s}d\tau=0.
\]
By a similar argument, one shows that%
\[
\left\Vert \boldsymbol{V}\left(  t_{2}-\tau\right)  F(u(\tau))\right\Vert
_{s-2\delta}\leq1+C_{0}\left(  \frac{1}{2\gamma(t_{2}-\tau)}\right)
^{\frac{1}{2}}L(\lVert u\left(  \tau\right)  \rVert_{s},0)\lVert u\left(
\tau\right)  \rVert_{s},
\]
and since
\begin{equation}
\Vert u(\tau)\Vert_{s}\leq\Vert u(\tau)-\boldsymbol{V}(\tau)f_{0}\Vert
_{s}+\Vert\boldsymbol{V}(\tau)f_{0}\Vert_{s}\leq M+\Vert f_{0}\Vert_{s},\text{
for all }\tau\in\lbrack0,T], \label{Eq-15A}%
\end{equation}
we have
\begin{gather}
\int\nolimits_{t_{1}}^{t_{2}}\left\Vert \boldsymbol{V}(t_{2}-\tau
)F(u(\tau))\right\Vert _{s}d\tau\label{Eq-16}\\
\leq L(M+\Vert f_{0}\Vert_{s},0)(M+\Vert f_{0}\Vert_{s})\left(  \int
\nolimits_{t_{1}}^{t_{2}}\left(  1+C_{0}\left(  \frac{1}{2\gamma(t_{2}-\tau
)}\right)  ^{\frac{1}{2}}\right)  d\tau\right) \nonumber\\
=L\left(  M+\Vert f_{0}\Vert_{s},0\right)  \left(  M+\Vert f_{0}(\cdot
)\Vert_{s}\right)  \left(  (t_{2}-t_{1})+C_{0}\left(  \sqrt{\frac
{2(t_{2}-t_{1})}{\gamma}}\right)  \right)  ,\nonumber
\end{gather}
and consequently, by applying the dominated convergence theorem,
\[
\lim_{t_{2}\rightarrow t_{1}}\int\nolimits_{t_{1}}^{t_{2}}\lVert
\boldsymbol{V}(t_{2}-\tau)F(u(\tau))\rVert_{s}d\tau=0.
\]

\textbf{Claim\ 2}. There\ exists\ $T_{0}$\ such\ that\ $\boldsymbol{N}%
(\mathcal{X}(M,T_{0},f_{0}))\subseteq\mathcal{X}(M,T_{0},f_{0})$.

By using a reasoning similar to the one used to established inequality
(\ref{Eq-16}), one gets%
\begin{gather*}
\Vert(\boldsymbol{N}u)(t)-\boldsymbol{V}(t)f_{0}\Vert_{s}\leq\int
\nolimits_{0}^{t}\Vert\boldsymbol{V}(t-\tau)F(u(\tau))\Vert_{s}d\tau\\
\leq L\left(  M+\Vert f_{0}\Vert_{s},0\right)  \left(  M+\Vert f_{0}\Vert
_{s}\right)  \left(  \int\nolimits_{0}^{t}\left(  1+C_{0}\left(  \frac
{1}{2\gamma(t-\tau)}\right)  ^{\frac{1}{2}}\right)  d\tau\right) \\
\leq L\left(  M+\Vert f_{0}\Vert_{s},0\right)  \left(  M+\Vert f_{0}\Vert
_{s}\right)  \left(  T+C_{0}\left(  \sqrt{\frac{2T}{\gamma}}\right)  \right)
.
\end{gather*}
Now taking $T_{0}$ such that
\begin{equation}
L\left(  M+\Vert f_{0}\Vert_{s},0\right)  \left(  M+\Vert f_{0}\Vert
_{s}\right)  \left(  T_{0}+C_{0}\left(  \sqrt{\frac{2T_{0}}{\gamma}}\right)
\right)  \leq M, \label{Eq-17}%
\end{equation}
we conclude that $\boldsymbol{N}u\in\mathcal{X}(M,T_{0},f_{0})$, for all
$u(t)\in\mathcal{X}(M,T_{0},f_{0})$.

\textbf{Claim\ 3.} There\ exists\ $T_{0}^{\prime}$%
\ such\ that\ $\boldsymbol{N}\ $is\ a\ contraction\ on$\ \mathcal{X}%
(M,T_{0}^{\prime},f_{0})$.

Given $u(t)$, $v(t)\in\mathcal{X}(M,T_{0},f_{0})$, by using Proposition
\ref{Prop2}, with
\[
C_{0}^{\prime}=L\left(  M+\Vert f_{0}\Vert_{s},M+\Vert f_{0}\Vert_{s}\right)
,
\]
see (\ref{Eq-15A}), we have%
\begin{align*}
\Vert\boldsymbol{N}u(t)-\boldsymbol{N}v(t)\Vert_{s}  &  \leq\int
\nolimits_{0}^{t}\Vert\boldsymbol{V}(t-\tau)[F(u(\tau))-F(v(\tau))]\Vert
_{s}d\tau\\
&  \leq\int\nolimits_{0}^{t}\left(  1+C_{0}\left(  \frac{1}{2\gamma(t-\tau
)}\right)  ^{\frac{1}{2}}\right)  \Vert F(u(\tau))-F(v(\tau))\Vert_{s-\alpha
}\text{ }d\tau\\
&  \leq C_{0}^{\prime}\int\nolimits_{0}^{t}\left(  1+C_{0}\left(  \frac
{1}{2\gamma(t-\tau)}\right)  ^{\frac{1}{2}}\right)  \Vert u(\tau)-v(\tau
)\Vert_{s}d\tau\\
&  \leq C_{0}^{\prime}\left(  \sup_{\tau\in\lbrack0,T_{0}]}\Vert
u(\tau)-v(\tau)\Vert_{s}\right)  \int\nolimits_{0}^{t}\left(  1+C_{0}\left(
\frac{1}{2\gamma(t-\tau)}\right)  ^{\frac{1}{2}}\right)  d\tau\\
&  \leq C_{0}^{\prime}\left(  T_{0}+C_{0}\left(  \sqrt{\frac{2T_{0}}{\gamma}%
}\right)  \right)  d(u(t),v(t)).
\end{align*}
Thus, taking $T_{0}^{\prime}$ such that
\begin{equation}
C:=C_{0}^{\prime}\left(  T_{0}^{\prime}+C_{0}\left(  \sqrt{\frac
{2T_{0}^{\prime}}{\gamma}}\right)  \right)  <1, \label{Eq-18}%
\end{equation}
we obtain that $d(\boldsymbol{N}u(t),\boldsymbol{N}v(t))\leq Cd(u(t),v(t))$,
that is, $\boldsymbol{N}$ is a strict contraction in $\mathcal{X}%
(M,T_{0}^{\prime},f_{0})$. We pick $T$ such that the inequalities
(\ref{Eq-17}) and (\ref{Eq-17}) hold true, and apply the Banach Fixed Point
Theorem to get $u(t)\in\mathcal{X}(M,T,f_{0})$ a unique fixed point of
$\boldsymbol{N}$, which satisfies the integral equation (\ref{Mild_Solution}),
where $T=T(\lVert f_{0}\rVert_{s},M)>0$.
\end{proof}

\begin{remark}
\label{Nota1}Let $\mathcal{X}$ be a Banach space and let $\boldsymbol{A}%
:Dom(\boldsymbol{A})\rightarrow\mathcal{X}$ be an operator with dense domain
such that $\boldsymbol{A}$ is the infinitesimal generator of a contraction
semigroup $\left(  \boldsymbol{S}_{t}\right)  _{t\geq0}$. Fix $T>0$ and let
$f:\left[  0,T\right]  \rightarrow\mathcal{X}$ be a continuous function.
Consider the Cauchy problem:%
\begin{equation}
\left\{
\begin{array}
[c]{l}%
u\in C\left(  \left[  0,T\right]  ,Dom(\boldsymbol{A})\right)  \cap
C^{1}\left(  \left[  0,T\right]  ,\mathcal{X}\right)  ;\\
\\
u_{t}=\boldsymbol{A}u+f(t),\ \ t\in\left[  0,T\right]  ;\\
\\
u(0)=u_{0}\in\mathcal{X}.
\end{array}
\right.  \label{Cauchy-Problem_2}%
\end{equation}
Then
\begin{equation}
u(t)=\boldsymbol{S}(t)u_{0}+\int\nolimits_{0}^{t}\boldsymbol{S}(t-\tau
)f(\tau))d\tau, \label{Mild_Solution_1}%
\end{equation}
for $t\in\left[  0,T\right]  $, see e.g. \cite[Lemma 4.1.1]{C-H}. Conversely,
if $u_{0}\in Dom(\boldsymbol{A})$, $f\in C\left(  \left[  0,T\right]
,\mathcal{X}\right)  $,
\[%
{\displaystyle\int\limits_{\left(  0,T\right)  }}
\left\Vert f\left(  \tau\right)  \right\Vert _{\mathcal{X}}\text{ }%
d\tau<\infty,
\]
then a solution of (\ref{Mild_Solution_1}) is a solution of the Cauchy problem
(\ref{Cauchy-Problem_2}), see e.g. \cite[Proposition 4.1.6]{C-H}.
\end{remark}

\begin{proposition}
\label{Prop4}The problem (\ref{Cauchy-Problem}) is equivalent to the integral
equation (\ref{Mild_Solution}). More precisely, if $s>n/2+2\delta$, and
$u(t)\in C([0,T];\mathcal{H}_{s})\cap C^{1}((0,T];\mathcal{H}_{s-2\delta})$ is
a solution of (\ref{Cauchy-Problem}), then $u(t)$ satisfies the integral
equation (\ref{Mild_Solution}). Conversely, if $s>n/2+2\delta$, and $u(t)\in
C([0,T];\mathcal{H}_{s})$ is a solution of (\ref{Mild_Solution}), then
$u(t)\in C^{1}([0,T];\mathcal{H}_{s-2\delta})$ and it satisfies
(\ref{Cauchy-Problem}).
\end{proposition}

\begin{proof}
It follows from Remark \ref{Nota1}, Propositions \ref{Prop3}, \ref{Prop2}, by
taking $\boldsymbol{A=}-\gamma\boldsymbol{D}_{x}^{\alpha}-\beta\boldsymbol{I}%
$, $Dom(\boldsymbol{A})=\mathcal{H}_{s}$, $\mathcal{X}=\mathcal{H}_{s-2\delta
}$, $f(t)=F(u\left(  t\right)  )$. We first recall that $\mathcal{D}%
\hookrightarrow\mathcal{H}_{s}\hookrightarrow\mathcal{H}_{s-2\delta}$, where
$\hookrightarrow$ means continuous embedding, \ an that $\mathcal{D}$ \ is
dense in $\mathcal{H}_{s-2\delta}$. If $u(t)$ is a solution of
(\ref{Cauchy-Problem}), then, since $F(u\left(  t\right)  )\in
C([0,T];\mathcal{H}_{s-2\delta})$, by Proposition \ref{Prop2}, $u\left(
t\right)  $ is a solution of (\ref{Mild_Solution}). Conversely, if $u\left(
t\right)  $ is a solution of (\ref{Mild_Solution}), since
\[%
{\displaystyle\int\limits_{\left(  0,T\right)  }}
\left\Vert F(u\left(  \tau\right)  )\right\Vert _{s-2\delta}\text{ }%
d\tau<\infty,
\]
by Proposition \ref{Prop2}, $u\left(  t\right)  $ is a solution of
(\ref{Cauchy-Problem}).
\end{proof}

\begin{lemma}
[{\cite[Theorem 5.1.1]{Milan}}]\label{Lemma_Gronwall}If $h\in L^{1}\left(
0,T\right)  $, with $T>0$, is real-valued function such that. If
\[
h(t)\leq a+b\int\nolimits_{0}^{t}h(s)ds,
\]
for $t\in$ $(0,T)$ a.e., where $a\in\mathbb{R}$ and $b\in\left[
0,\infty\right)  $ then $h(t)\leq ae^{bt}$ for almost all $t$ in $(0,T)$.
\end{lemma}

\begin{proposition}
\label{Prop5}Let $f_{0}$, $f_{1}\in\mathcal{H}_{s}$ and $u(t),v(t)\in
C[0,T];\mathcal{H}_{s})$ be the corresponding solutions of equation
(\ref{Mild_Solution}) with initial conditions $u(0)=f_{0}$ and $v(0)=f_{1}$,
respectively. If $s>n/2+2\delta$, then
\[
\Vert u(t)-v(t)\Vert_{s}\leq e^{L\left(  W,W\right)  }\Vert f_{0}-f_{1}%
\Vert_{s},
\]
where $L$ is given in Proposition \ref{Prop1} and
\[
W:=\max\left\{  \sup_{t\in\lbrack0,T]}\Vert u(t)\Vert_{s},\sup_{t\in
\lbrack0,T]}\Vert v(t)\Vert_{s}\right\}  .
\]

\end{proposition}

\begin{proof}
By using (\ref{Mild_Solution}), we have
\[
u(t)-v(t)=\boldsymbol{V}(t)(f_{0}-f_{1})+\int\nolimits_{0}^{t}\boldsymbol{V}%
(t-\tau)\{F(u(\tau))-F(v(\tau))\}d\tau.
\]
By using Proposition \ref{Prop1}, we get%
\begin{gather*}
\Vert u(t)-v(t)\Vert_{s}\leq\Vert f_{0}-f_{1}\Vert_{s}+\int\nolimits_{0}%
^{t}\Vert\boldsymbol{V}(t-\tau)\{F(u(\tau))-F(v(\tau))\}\Vert_{s}d\tau\\
\leq\Vert f_{0}-f_{1}\Vert_{s}+\int\nolimits_{0}^{t}\Vert F(u(\tau
))-F(v(\tau))\Vert_{s-\alpha}\text{ }d\tau\\
\leq\Vert f_{0}-f_{1}\Vert_{s}+L(W,W)\int\nolimits_{0}^{t}\Vert u(\tau
)-v(\tau)\Vert_{s}d\tau.
\end{gather*}
Now the result follow from \ Lemma \ref{Lemma_Gronwall}, by taking
\ $h(t)=\lVert u(t)-v(t)\rVert_{s}$, $a=\lVert f_{0}-f_{1}\rVert_{s}$,
$b=L(W,W)$.
\end{proof}

\begin{proposition}
\label{Prop6}Let $s>n/2+2\delta$ and $\delta\geq0$. Then, the map
$f_{0}\mapsto u(t)$ is continuous in the following sense: if $f_{0}%
^{(n)}\rightarrow f_{0}$ in $\mathcal{H}_{s}$ and $u_{n}(t)\in C\left(
\left[  0,T_{n}\right]  ;\mathcal{H}_{s}\right)  $, with $T_{n}=T\left(
\left\Vert f_{0}^{(n)}\right\Vert _{s},M\right)  >0$, are the corresponding
solutions to the Cauchy problem (\ref{Cauchy-Problem}) with $u_{n}%
(0)=f_{0}^{(n)}$. Then, there exist $T>0$ and a positive integer
$N=N(\gamma,f_{0},T)$ such that $T_{n}\geq T$ for all $n\geq N$ and
\begin{equation}
\lim_{n\rightarrow\infty}\sup_{t\in\lbrack0,T]}\left\Vert u_{n}%
(t)-u(t)\right\Vert _{s}=0. \label{Eq-20}%
\end{equation}

\end{proposition}

\begin{proof}
By Proposition \ref{Prop3}, the $T_{n}=T\left(  \left\Vert f_{0}%
^{(n)}\right\Vert _{s},M\right)  >0$ are continuous functions of $\left\Vert
f_{0}^{(n)}\right\Vert _{s}$, then, given $T^{\ast}>0$ there exists
$N\in\mathbb{N}$ such that $T^{\ast}\leq T_{n}$ for all $n\geq N$. We set
$T:=\min\left\{  T^{\ast},T_{1},T_{2},\ldots,T_{N-1}\right\}  $. Therefore,
all the $u_{n}(t)$ are defined on $[0,T]$, furthermore, $u\in\mathcal{X}%
\left(  M,T,f_{0}^{(n)}\right)  $ for all $n$, and
\[
\left\Vert u_{n}(t)\right\Vert _{s}\leq\left\Vert f_{0}^{(n)}\right\Vert
_{s}+M\leq\delta+M,
\]
where $\delta=\sup_{n\in\mathbb{N}}\left\Vert f_{0}^{(n)}\right\Vert _{s}$.
Now
\[
\sup_{t\in\lbrack0,T]}\left\Vert u_{n}(t)\right\Vert _{s}\leq\delta+M\text{
for all }n\text{, \ and}\sup_{t\in\lbrack0,T]}\Vert u(t)\Vert_{s}\leq
\delta+M.
\]

On the other hand, by reasoning as in the proof of Proposition \ref{Prop5}, we
have%
\[
\left\Vert u_{n}(t)-u(t)\right\Vert _{s}\leq\left\Vert f_{0}^{\left(
n\right)  }-f_{0}\right\Vert _{s}+L(\delta+M,\delta+M)\int\limits_{0}^{t}\Vert
u_{n}(\tau)-u(\tau)\Vert_{s}d\tau,
\]
and by applying Lemma \ref{Lemma_Gronwall}
\[
\left\Vert u_{n}(t)-u(t)\right\Vert _{s}\leq e^{TL(\delta+M,\delta
+M)}\left\Vert f_{0}^{\left(  n\right)  }-f_{0}\right\Vert _{s},
\]
which in turns implies (\ref{Eq-20}).
\end{proof}

\subsection{Proof of the Main result}

The local well-posedness\ of the Cauchy problem (\ref{Cauchy-Problem}) in
$\mathcal{H}_{s}$, $s>n/2+2\delta$, follows from Propositions \ref{Prop3},
\ref{Prop5}, \ref{Prop6}.

\section{\label{Section_5}The Blow-up phenomenon}

In this section, we study the blow-up phenomenon for the solution of the
equation%
\begin{equation}
\left\{
\begin{array}
[c]{ll}%
u_{t}=-\gamma\boldsymbol{D}_{x}^{\alpha}u+F(u)+\boldsymbol{D}_{x}^{\alpha_{1}%
}u^{3}, & x\in\mathbb{Q}_{p}^{n},\ t\in\left[  0,T\right]  ;\\
& \\
u(0)=f_{0}\in\mathcal{H}_{\infty}, &
\end{array}
\right.  \label{5.1}%
\end{equation}
\textup{where }$F(u)=-u^{3}+\left(  \beta+1\right)  u^{2}-\beta u$. We will
say that a non-negative solution $u(x,t)\geq0$ of (\ref{5.1}) blow-up in a
finite time $T>0$, if $\lim_{t\rightarrow T^{-}}\sup_{x\in\mathbb{Q}_{p}^{n}%
}u(x,t)=+\infty$. This limit makes sense since $\mathcal{H}_{\infty
}(\mathbb{Q}_{p}^{n},\mathbb{C})$ is continuously embedded in $C_{0}%
(\mathbb{Q}_{p}^{n},\mathbb{C})$, \cite[Theorem 10.15 ]{KKZuniga}.

\subsection{$p-$adic wavelets and pseudo-differential operators}

We denote by $C(\mathbb{Q}_{p},\mathbb{C})$ the $\mathbb{C}-$vector space of
\ continuous $\mathbb{C}-$valued functions defined on $\mathbb{Q}_{p}$.

We fix a function $\mathfrak{a}:\mathbb{R}_{+}\rightarrow\mathbb{R}_{+}$ and
define the pseudo-differential operator%
\[%
\begin{array}
[c]{ccc}%
\mathcal{D} & \rightarrow & C(\mathbb{Q}_{p},\mathbb{C})\cap L^{2}\\
&  & \\
\varphi & \rightarrow & \boldsymbol{A}\varphi,
\end{array}
\]
where $\left(  \boldsymbol{A}\varphi\right)  \left(  x\right)  =\mathcal{F}%
_{\xi\rightarrow x}^{-1}\left\{  \mathfrak{a}\left(  \left\vert \xi\right\vert
_{p}\right)  \mathcal{F}_{x\rightarrow\xi}\varphi\right\}  $.

The set of functions $\left\{  \Psi_{rnj}\right\}  $ defined as%
\begin{equation}
\Psi_{rnj}\left(  x\right)  =p^{\frac{-r}{2}}\chi_{p}\left(  p^{-1}j\left(
p^{r}x-n\right)  \right)  \Omega\left(  \left\vert p^{r}x-n\right\vert
_{p}\right)  , \label{eq4}%
\end{equation}
where $r\in\mathbb{Z}$, $j\in\left\{  1,\cdots,p-1\right\}  $, and $n$ runs
through a fixed set of representatives of $\mathbb{Q}_{p}/\mathbb{Z}_{p}$, is
an orthonormal basis of $L^{2}(\mathbb{Q}_{p})$ consisting of eigenvectors of
operator $\boldsymbol{A}$:%
\begin{equation}
\boldsymbol{A}\Psi_{rnj}=\mathfrak{a}(p^{1-r})\Psi_{rnj}\text{ for any
}r\text{, }n\text{, }j\text{,} \label{eq5}%
\end{equation}
see e.g. \cite[Theorem 3.29]{KKZuniga}, \cite[Theorem 9.4.2]{A-K-S}.\ Notice
that%
\[
\widehat{\Psi}_{rnj}\left(  \xi\right)  =p^{\frac{r}{2}}\chi_{p}\left(
p^{-r}n\xi\right)  \Omega\left(  \left\vert p^{-r}\xi+p^{-1}j\right\vert
_{p}\right)  ,
\]
and then%
\[
\mathfrak{a}\left(  \left\vert \xi\right\vert _{p}\right)  \widehat{\Psi
}_{rnj}\left(  \xi\right)  =\mathfrak{a}(p^{1-r})\widehat{\Psi}_{rnj}\left(
\xi\right)  .
\]

In particular, $\boldsymbol{D}_{x}^{\alpha}\Psi_{rnj}=p^{\left(  1-r\right)
\alpha}\Psi_{rnj}$, for any $r,n,j$ and $\alpha>0$, and since $p^{\left(
1-r\right)  \alpha}$,%
\[
\boldsymbol{D}_{x}^{\alpha}\operatorname{Re}\left(  \Psi_{rnj}\right)
=p^{\left(  1-r\right)  \alpha}\operatorname{Re}\left(  \Psi_{rnj}\right)
\text{, }\boldsymbol{D}_{x}^{\alpha}\operatorname{Im}\left(  \Psi
_{rnj}\right)  =p^{\left(  1-r\right)  \alpha}\operatorname{Im}\left(
\Psi_{rnj}\right)  .
\]
And,%
\begin{align*}
\left\{  \Psi_{rn1}\left(  x\right)  \right\}  ^{2}  &  =p^{-r}\chi_{p}\left(
2p^{-1}\left(  p^{r}x-n\right)  \right)  \Omega\left(  \left\vert
p^{r}x-n\right\vert _{p}\right) \\
&  =p^{r}\left\{  \Psi_{rn1}\left(  x\right)  \right\}  ^{2}=p^{\frac{r}{2}%
}\Psi_{rn2}\left(  x\right)  ,
\end{align*}
then
\[
\boldsymbol{D}_{x}^{\alpha}\operatorname{Re}\left(  \left\{  \Psi_{rn1}\left(
x\right)  \right\}  ^{2}\right)  =p^{\frac{r}{2}}p^{\left(  1-r\right)
\alpha}\operatorname{Re}\left(  \Psi_{rn2}(x)\right)  =p^{\left(  1-r\right)
\alpha}\operatorname{Re}\left(  \left\{  \Psi_{rn1}\left(  x\right)  \right\}
^{2}\right)  .
\]

\subsection{The blow-up}

In this section, we assume that $u(x,t)$ is real-valued non-negative solution
of the Cauchy problem (\ref{Cauchy-Problem}) in $\mathcal{H}_{\infty}$. We set
$w(x):=\operatorname{Re}\left(  \left\{  \Psi_{rn1}\left(  x\right)  \right\}
^{2}\right)  $, so $\boldsymbol{D}_{x}^{\alpha}w(x)=p^{\left(  1-r\right)
\alpha}w(x)$. Thus $w(x)dx$ defines a (positive) measure. We also set
$G(t):=\int_{\mathbb{Q}_{p}}u(x,t)w(x)dx$, where $u(x,t)$ is a positive
solution of (\ref{5.1}), then%
\begin{gather}
G^{\prime}(t)=\int\limits_{\mathbb{Q}_{p}}u_{t}(x,t)w(x)dx=-\gamma
\int\limits_{\mathbb{Q}_{p}}(\boldsymbol{D}_{x}^{\alpha}%
u)(x,t)w(x)dx\nonumber\\
+\int\limits_{\mathbb{Q}_{p}}F(u(x,t))w(x)dx+\int\limits_{\mathbb{Q}_{p}%
}(\boldsymbol{D}_{x}^{\alpha_{1}}u^{3})(x,t)w(x)dx. \label{Identity_2}%
\end{gather}
Now, by using that $\boldsymbol{D}_{x}^{\alpha}u(\cdot,t)$, $w\in L^{2}$, and
$F(u(\cdot,t))$, $\boldsymbol{D}_{x}^{\alpha_{1}}u^{3}(\cdot,t)\in L^{2}$
since for $s>n/2$, $\mathcal{H}_{s}$ is a Banach algebra contained in $L^{2}%
$\ cf. Proposition \ref{Prop1}, and applying the Parseval-Steklov theorem, we
get (\ref{Identity_2}) can be rewritten as
\[
G^{\prime}(t)=\int\limits_{\mathbb{Q}_{p}}\left(  -\gamma p^{\left(
1-r\right)  \alpha}u(x,t)+F(u(x,t))+p^{\left(  1-r\right)  \alpha_{1}}%
u^{3}(x,t)\right)  w(x)dx.
\]
Since the function $H(y)=-\gamma p^{\left(  1-r\right)  \alpha}%
y+F(y)+p^{\left(  1-r\right)  \alpha_{1}}y^{3}$ is convex because
\[
H^{\prime\prime}(y)=-6y+2\left(  \beta+1\right)  +p^{\left(  1-r\right)
\alpha_{1}}6y=6y\left(  p^{\left(  1-r\right)  \alpha_{1}}-1\right)  +2\left(
\beta+1\right)  \geq0,
\]
for $y\geq0$, and $r\leq0$, we can use the Jensen's inequality to get
$G^{\prime}(t)\geq H(G(t))$, then the function $G(t)$ can not remain finite
for every $t\in\lbrack0,\infty)$. Then there exists $T\in(0,\infty)$ such that
$\lim_{t\rightarrow T^{-}}G(t)=+\infty$, hence $u(x,t)$ blow ups at the time
$T$. Then we have established the following result:

\begin{theorem}
\label{Thorem2}Let $u(x,t)$ be a positive solution of (\ref{5.1}). Then there
$T\in\left(  0,+\infty\right)  $ depending on the initial datum such that
$\lim_{t\rightarrow T^{-}}\sup_{x\in\mathbb{Q}_{p}^{n}}u(x,t)=+\infty$.
\end{theorem}

\section{\label{Section_6}Numerical Simulations}

In this section, we present two numerical simulations for the solution of
problem (\ref{5.1}) (in dimension one) for a suitable initial datum. We solve
and visualize (using a heat map) the radial profiles of the solution of
(\ref{5.1}). We consider equation (\ref{5.1}) for radial functions
$u(x,\cdot)$. In \cite{Kochubei-1}, Kochubei obtained a formula for
$\boldsymbol{D}_{x}^{\alpha}u\left(  x,t\right)  $ as an absolutely convergent
real series, we truncate this series and then we apply the classic Euler
Forward Method (see e.g. \cite{Euler}) to find the values of $u(p^{-ord(x)}%
,t)$, when $-20\leq ord(x)\leq20$ (vertical axis) and when $t=\{t_{k}%
:\,\,t_{k}=1/k,k=1,\dots,300\}$ (horizontal axis). \ In Figure 1, on the left,
the heat map of the numerical solution of the homogeneous equation
$u_{t}(x,t)=-\boldsymbol{D}_{x}^{\alpha}u(x,t)$ with initial data
$u(x,0)=4e^{-p^{|ord(x)|}/100}$ (Gaussian bell type), and parameters $p=3$,
$\alpha=0.2$, $\gamma=1$. On the right side, we have the numerical solution of
the equation $u_{t}(x,t)=-\boldsymbol{D}_{x}^{\alpha}u(x,t)-u^{3}%
(x,t)+(\beta+1)u^{2}(x,t)-\beta u(x,t)+\boldsymbol{D}_{x}^{\alpha_{1}}%
u^{3}(x,t)$, with $p=3$, $\alpha=0.2$, $\alpha_{1}=0.1$, and $\beta=0.7$. 

\begin{figure}[pth]
\hskip 4cm \epsfxsize=6cm \epsfbox{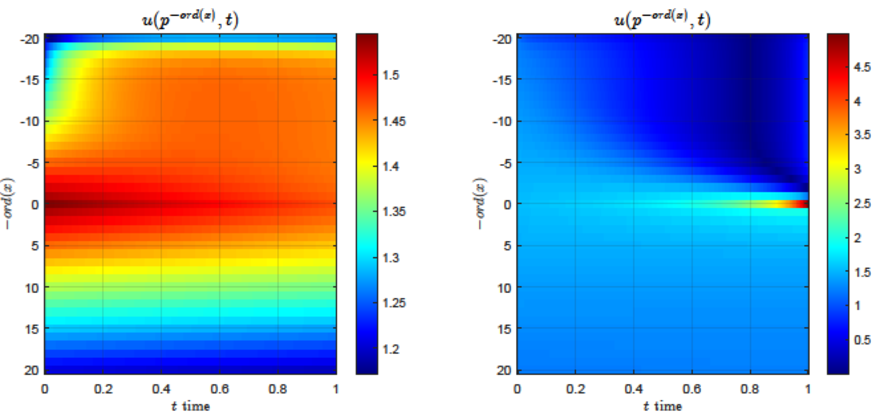}\end{figure}

On the left side of the Figure 1, we observe that the solution $u$ is
uniformly decreasing with respect to the variable $t$. This behavior is
typical for solutions of diffusion equations. These equations have been
extensively studied, see e.g. \cite{KKZuniga}, \cite{Zuniga-LNM-2016} and the
references therein.

On the right side of Figure 1, we see that the evolution of $u(x,t)$ is
controlled by the diffusion term $-\boldsymbol{D}_{x}^{\alpha}u(x,t)$, up to a
time $T$ (blow-up time), this behavior is similar to that described above.
When $t>T$, the reactive term $-u^{3}(x,t)+(\beta+1)u^{2}(x,t)-\beta
u(x,t)+\boldsymbol{D}_{x}^{\alpha_{1}}u^{3}(x,t)$ takes over and $u(x,t)$
grows rapidly towards infinity.

The method converges quite fast, but still lacks a mathematical formalism to
support it, for this reason we refer to it as a numerical simulation of the solution.


\begin{thebibliography}{99}                                                                                               %


\bibitem {A-K-S}Albeverio S., Khrennikov A. Yu., Shelkovich V. M., Theory of
$p-$adic distributions: linear and nonlinear models. London Mathematical
Society Lecture Note Series, 370. Cambridge University Press, Cambridge, 2010.

\bibitem {A-K-S-2}Albeverio S., Khrennikov A. Yu., Shelkovich V. M., The
Cauchy problems for evolutionary pseudo-differential equations over $p-$adic
field and the wavelet theory, J. Math. Anal. Appl. 375 (2011), no. 1, 82--98.

\bibitem {C-H}Cazenave Thierry, Haraux Alain, An introduction to semilinear
evolution equations. Oxford University Press, 1998.

\bibitem {Chacon et al}Chac\'{o}n-Cort\'{e}s L. F., Guti\'{e}rrez-Garc\'{\i}a
Ismael, Torresblanca-Badillo Anselmo, Vargas Andr\'{e}s Finite time blow-up
for a p-adic nonlocal semilinear ultradiffusion equation, J. Math. Anal. Appl.
494 (2021), no. 2, Paper No. 124599, 22 pp.

\bibitem {Chacon-Zuniga 1}Chac\'{o}n-Cort\'{e}s L. F., Z\'{u}\~{n}iga-Galindo
W. A., Non-local operators, non-Archimedean parabolic-type equations with
variable coefficients and Markov processes, Publ. Res. Inst. Math. Sci. 51
(2015), no. 2, 289--317.

\bibitem {Chacon-Zuniga 2}Chac\'{o}n-Cort\'{e}s L. F., Z\'{u}\~{n}iga-Galindo
W. A. Nonlocal operators, parabolic-type equations, and ultrametric random
walks, J. Math. Phys. 54 (2013), no. 11, 113503, 17 pp.

\bibitem {De la Cruz}De la Cruz Richard, Lizarazo Vladimir , Local
well-posedness to the Cauchy problem for an equation of Nagumo type. Preprint 2019.

\bibitem {Gelfand-Vilenkin}Gel'fand I.M., Vilenkin N.Y., Generalized
Functions. Applications of Harmonic Analysis, vol. 4. Academic Press, New
York, 1964.

\bibitem {Gorka et al 1}G\'{o}rka Przemys\l aw, Kostrzewa Tomasz, Reyes
Enrique G., Sobolev spaces on locally compact abelian groups: compact
embeddings and local spaces, J. Funct. Spaces 2014, Art. ID 404738, 6 pp.

\bibitem {Gorka et al 2}G\'{o}rka Przemys\l aw, Kostrzewa Tomasz, Sobolev
spaces on metrizable groups, Ann. Acad. Sci. Fenn. Math. 40 (2015), no. 2, 837--849.

\bibitem {Halmos}Halmos Paul R., Measure Theory. D. Van Nostrand Co., Inc.,
New York, N.Y., 1950.

\bibitem {Haran}Haran S.,\ Quantizations and symbolic calculus over the
$p-$adic numbers. Ann. Inst. Fourier 43 (1993), no. 4, 997--1053.

\bibitem {Kaneko}Kaneko H., Besov space and trace theorem on a local field and
its application, Math. Nachr. 285 (2012), no. 8-9, 981--996.

\bibitem {Kochubei}Kochubei A.N., Pseudo-Differential Equations and
Stochastics over Non-Archimedean Fields. Marcel Dekker, New York, 2001.

\bibitem {Kochubei-1}Kochubei A. N., Radial solutions of non-Archimedean
pseudodifferential equations, Pacific J. Math. 269 (2014), no. 2, 355--369.

\bibitem {Kochubei pacif}Kochubei A.N., A non-Archimedean wave equation,
Pacific J. Math. 235 (2008), no. 2, 245--261.

\bibitem {Khrennikov-Kochubei}Khrennikov Andrei Yu, Kochubei Anatoly N.,
$p-$Adic Analogue of the Porous Medium Equation, J Fourier Anal Appl (2018) 24:1401--1424.

\bibitem {KKZuniga}Khrennikov Andrei Yu., Kozyrev Sergei V.,
Z\'{u}\~{n}iga-Galindo W. A., Ultrametric pseudodifferential equations and
applications. Encyclopedia of Mathematics and its Applications, 168. Cambridge
University Press, Cambridge, 2018.

\bibitem {Khrennikov et al 1}Khrennikov Andrei, Oleschko Klaudia, Correa
L\'{o}pez,Maria de Jes\'{u}s, Application of p-adic wavelets to model
reaction-diffusion dynamics in random porous media, J. Fourier Anal. Appl. 22
(2016), no. 4, 809--822.

\bibitem {Milan}Miklav\v{c}i\v{c} Milan. Applied functional analysis and
partial differential equations. World Scientific Publishing Co., Inc., River
Edge, NJ, 1998.

\bibitem {Nagumo et al}Nagumo J., Yoshizawa S. and Arimoto S. Bistable
Transmission Lines. IEEE Transactions on Circuit Theory, vol. 12, no. 3, pp.
400-412, September 1965.

\bibitem {Olesko-Khrennikov}Oleschko K., Khrennikov A., Transport through a
network of capillaries from ultrametric diffusion equation with quadratic
nonlinearity, Russ. J. Math. Phys. 24 (2017), no. 4, 505--516.

\bibitem {Euler}Press W. H., Flannery B. P., Teukolsky, S. A., and Vetterling
W. T., Numerical Recipes in FORTRAN: The Art of Scientific Computing, 2nd ed.
Cambridge, England: Cambridge University Press, p. 710, 1992.

\bibitem {Pourhadi et al}Pourhadi Ehsan, Khrennikov Andrei Yu., Oleschko
Klaudia, Correa Lopez Mar\'{\i}a de Jes\'{u}s, Solving nonlinear $p-$adic
pseudo-differential equations: combining the wavelet basis with the Schauder
fixed point theorem, J. Fourier Anal. Appl. 26 (2020), no. 4, Paper No. 70, 23 pp.

\bibitem {Rodriguez-Zuniga}Rodr\'{\i}guez-Vega J. J., Z\'{u}\~{n}iga-Galindo
W. A., Elliptic pseudodifferential equations and Sobolev spaces over $p-$adic
fields, Pacific J. Math. 246 (2010), no. 2, 407--420.

\bibitem {Taibleson}Taibleson M.H., Fourier Analysis on Local Fields.
Princeton University Press, Princeton, 1975.

\bibitem {Torresblanca-Zuniga 1}Torresblanca-Badillo Anselmo,
Z\'{u}\~{n}iga-Galindo W. A., Ultrametric diffusion, exponential landscapes,
and the first passage time problem, Acta Appl. Math. 157 (2018), 93--116.

\bibitem {Torresblanca-Zuniga 2}Torresblanca-Badillo Anselmo,
Z\'{u}\~{n}iga-Galindo W. A., Non-Archimedean pseudodifferential operators and
Feller semigroups, $p-$Adic Numbers Ultrametric Anal. Appl. 10 (2018), no. 1, 57--73.

\bibitem {V-V-Z}Vladimirov V. S., Volovich I. V. and Zelenov E. I., $p-$adic
analysis and mathematical physics, World Scientific, 1994.

\bibitem {Zuniga-1}Zambrano-Luna B., Z\'{u}\~{n}iga-Galindo W. A., $p-$Adic
Cellular Neural Networks. https://arxiv.org/abs/2107.07980.\ 

\bibitem {Zuniga-JMAA}Z\'{u}\~{n}iga-Galindo W. A., Reaction-diffusion
equations on complex networks and Turing patterns, via $p-$adic analysis, J.
Math. Anal. Appl. 491 (2020), no. 1, 124239, 39 pp.

\bibitem {Zuniga-2}Z\'{u}\~{n}iga-Galindo W. A., Non-archimedean replicator
dynamics and Eigen's paradox, J. Phys. A 51 (2018), no. 50, 505601, 26 pp.

\bibitem {Zuniga-Nonlinearity}Z\'{u}\~{n}iga-Galindo W. A., Non-Archimedean
reaction-ultradiffusion equations and complex hierarchic systems, Nonlinearity
31 (2018), no. 6, 2590--2616.

\bibitem {Zuniga-JFAA}Z\'{u}\~{n}iga-Galindo W. A. Non-Archimedean white
noise, pseudodifferential stochastic equations, and massive Euclidean fields,
J. Fourier Anal. Appl. 23 (2017), no. 2, 288--323.

\bibitem {Zuniga-LNM-2016}Z\'{u}\~{n}iga-Galindo W. A., Pseudodifferential
equations over non-Archimedean spaces. Lecture Notes in Mathematics, 2174.
Springer, Cham, 2016.

\bibitem {Zuniga-3}Z\'{u}\~{n}iga-Galindo W. A., The Cauchy problem for
non-Archimedean pseudodifferential equations of Klein-Gordon type, J. Math.
Anal. Appl. 420 (2014), no. 2, 1033--1050.

\bibitem {Zuniga-4}Z\'{u}\~{n}iga-Galindo W. A., Parabolic equations and
Markov processes over $p-$adic fields. Potential Anal. 28 (2008), no. 2, 185--200.

\bibitem {Zuniga-5}Zuniga-Galindo W. A., Fundamental solutions of
pseudo-differential operators over $p-$adic fields. Rend, Sem. Mat. Univ.
Padova 109 (2003), 241--245.
\end{thebibliography}
\end{document}